\documentclass[preprint]{elsarticle}

\usepackage{amssymb}
\usepackage{amsmath} 
\usepackage{amsfonts}
\usepackage{amsthm}
\usepackage[colorlinks]{hyperref} 

\newtheorem{thm}{Theorem}[section]
\newtheorem{cor}[thm]{Corollary}
\newtheorem{lem}[thm]{Lemma}

\theoremstyle{definition}
\newtheorem{defn}[thm]{Definition}
\newtheorem{ex}[thm]{Example}
\newtheorem{remk}[thm]{Remark}

\newcommand{\HH}{\mathcal{H}}

\newcommand{\R}{\mathbb{R}}

\newcommand{\leqs}{\leqslant}
\newcommand{\geqs}{\geqslant}

\newcommand{\ap}{\alpha}

\newcommand{\ep}{\epsilon}
\newcommand{\ld }{\lambda}

\newcommand{\sm}{\,\sigma\,}

\numberwithin{equation}{section}

\DeclareMathOperator{\Sp}{Sp}
\DeclareMathOperator{\Range}{Range}

\begin{document}

\begin{frontmatter}

\title{Cancellability and Regularity of Operator Connections}

\author{Pattrawut Chansangiam} 
\ead{kcpattra@kmitl.ac.th}

\address{Department of Mathematics, Faculty of Science, 
	King Mongkut's Institute of Technology Ladkrabang, 
	Bangkok 10520, Thailand.}

\begin{abstract}

An operator connection is a binary operation assigned to each pair of positive operators
satisfying monotonicity, continuity from above and the transformer inequality.
In this paper, we introduce and characterize the concepts of cancellability and regularity of operator connections with respect to operator monotone functions, Borel measures and certain operator equations.
In addition, we investigate the existence and the uniqueness of solutions for such operator equations. 

\end{abstract}

\begin{keyword}
operator connection \sep operator monotone function \sep operator mean
\MSC[2010]47A63 \sep 47A64  
\end{keyword}

\end{frontmatter}

\section{Introduction}

A general theory of connections and means for positive operators
was given by Kubo and Ando \cite{Kubo-Ando}.
Let $B(\HH)$ be the algebra of bounded linear operators on
a Hilbert space $\HH$. The set of positive operators on $\HH$ is denoted by $B(\HH)^+$.
Denote the spectrum of an operator $X$ by $\Sp(X)$.
For Hermitian operators $A,B \in B(\HH)$,
the partial order $A \leqs B$ means that $B-A \in B(\HH)^+$.
The notation $A>0$ suggests that $A$ is a strictly positive operator.
A \emph{connection} is a binary operation $\sm$ on $B(\HH)^+$
such that for all positive operators $A,B,C,D$:
\begin{enumerate}
	\item[(M1)] \emph{monotonicity}: $A \leqs C, B \leqs D \implies A \sm B \leqs C \sm D$
	\item[(M2)] \emph{transformer inequality}: $C(A \sm B)C \leqs (CAC) \sm (CBC)$
	\item[(M3)] \emph{continuity from above}:  for $A_n,B_n \in B(\HH)^+$,
                if $A_n \downarrow A$ and $B_n \downarrow B$,
                 then $A_n \sm B_n \downarrow A \sm B$.
                 Here, $A_n \downarrow A$ indicates that $A_n$ is a decreasing sequence
                 and $A_n$ converges strongly to $A$.
\end{enumerate}
Two trivial examples are the left-trivial mean $\omega_l : (A,B) \mapsto A$ and the right-trivial mean $\omega_r: (A,B) \mapsto B$.
Typical examples of a connection are the sum $(A,B) \mapsto A+B$ and the parallel sum
\begin{align*}
    A \,:\,B = (A^{-1}+B^{-1})^{-1}, \quad A,B>0,
\end{align*}
the latter being introduced by Anderson and Duffin \cite{Anderson-Duffin}.
From the transformer inequality, every connection is \emph{congruence invariant} in the sense that for each $A,B \geqs 0$
and $C>0$ we have
\begin{align*}
	C(A \sm B)C \:=\: (CAC) \sm (CBC).
\end{align*}

A \emph{mean} is a connection $\sigma$ with normalized condition $I \sm I = I$ or, equivalently,
fixed-point property $A \sm A =A$ for all $A \geqs 0$.
The class of Kubo-Ando means cover many well-known operator means
in practice, e.g.
	\begin{itemize}
		\item	$\ap$-weighted arithmetic means: $A \triangledown_{\ap} B = (1-\ap)A + \ap B$
		\item	$\ap$-weighted geometric means:
    			$A \#_{\ap} B =  A^{1/2}
    			({A}^{-1/2} B  {A}^{-1/2})^{\ap} {A}^{1/2}$
		\item	$\ap$-weighted harmonic means: $A \,!_{\ap}\, 
				B = [(1-\ap)A^{-1} + \ap B^{-1}]^{-1}$ 
		\item	logarithmic mean: $(A,B) \mapsto A^{1/2}f(A^{-1/2}BA^{-1/2})A^{1/2}$ 
		where $f: \R^+ \to \R^+$, $f(x)=(x-1)/\log{x}$, $f(0) \equiv 0$ and $f(1) \equiv 1$.
		Here, $\R^+=[0, \infty)$.
	\end{itemize}

It is a fundamental that there are one-to-one correspondences between the following objects:
\begin{enumerate}
	\item[(1)]	operator connections on $B(\HH)^+$
	\item[(2)]	operator monotone functions from $\R^+$ to $\R^+$
	\item[(3)]	finite (positive) Borel measures on $[0,1]$
	\item[(4)]	monotone (Riemannian) metrics on the smooth manifold of positive definite matrices. 
\end{enumerate}
Recall that a function $f: \R^+ \to \R^+$ is said to be \emph{operator monotone} if
\begin{align*}
	A \leqs B \implies f(A) \leqs f(B)
\end{align*}
for all positive operators $A,B \in B(\HH)$ and for all Hilbert spaces $\HH$.
This concept was introduced in \cite{Lowner}; see also \cite{Bhatia,Hiai,Hiai-Yanagi}.
A remarkable fact is that (see \cite{Hansen-Pedersen}) a function $f: \R^+ \to \R^+$
is operator monotone if and only if it is \emph{operator concave}, i.e.
\begin{align*}
	f((1-\ap)A + \ap B) \:\geqs\: (1-\ap) f(A) + \ap f(B), \quad \ap \in (0,1)
\end{align*}
holds for all positive operators $A,B \in B(\HH)$ and for all Hilbert spaces $\HH$.


A connection $\sigma$ on $B(\HH)^+$ can be characterized
via operator monotone functions as follows:

\begin{thm}[\cite{Kubo-Ando}] \label{thm: Kubo-Ando f and sigma}
	Given a connection $\sigma$, there is a unique operator monotone function $f: \R^+ \to \R^+$ satisfying
				\begin{align*}
    			f(x)I = I \sm (xI), \quad x \geqs 0.
				\end{align*}
	Moreover, the map $\sigma \mapsto f$ is a bijection.
\end{thm}

We call $f$ the \emph{representing function} of $\sigma$.
A connection also has a canonical characterization with respect to a Borel measure 
via a meaningful integral representation as follows.

\begin{thm}[\cite{Pattrawut}] \label{thm: 1-1 conn and measure}
Given a finite Borel measure $\mu$ on $[0,1]$, the binary operation
	\begin{align}
		A \sm B = \int_{[0,1]} A \,!_t\, B \,d \mu(t), \quad A,B \geqs 0
		\label{eq: int rep connection}
	\end{align}
	is a connection on $B(\HH)^+$.
	Moreover, the map $\mu \mapsto \sigma$ is bijective,
	in which case the representing function of $\sigma$ is given by
	\begin{align}
		f(x) = \int_{[0,1]} (1 \,!_t\, x) \,d \mu(t), \quad x \geqs 0. \label{int rep of OMF}
	\end{align}
\end{thm}
We call $\mu$ the \emph{associated measure} of $\sigma$.
A connection is a mean if and only if $f(1)=1$ or its associated measure is a probability measure.
Hence every mean can be regarded as an average of weighted harmonic means.
From \eqref{eq: int rep connection} and \eqref{int rep of OMF}, $\sigma$
and $f$ are related by
\begin{align}
	f(A) \:=\: I \sm A, \quad A \geqs 0.  \label{eq: f(A) = I sm A}
\end{align}
A connection $\sigma$ is said to be \emph{symmetric} if $A \sm B = B \sm A$ for all $A,B \geqs 0$. 

The notion of monotone metrics arises naturally in quantum mechanics.
A metric on the differentiable manifold of $n$-by-$n$ positive definite matrices 
is a continuous family of positive definite sesqilinear forms assigned to each invertible density matrix
in the manifold. 
A monotone metric is a metric with contraction property under stochastic maps.
It was shown in \cite{Petz} that there is a one-to-one correspondence between
operator connections and monotone metrics.
Moreover, symmetric metrics correspond to symmetric means.
In \cite{Hansen}, the author defined a symmetric metric to be \emph{nonregular} 
if $f(0)=0$ where $f$ is the associated operator monotone function.  
In \cite{Gibilisco}, $f$ is said to be \emph{nonregular} if $f(0) = 0$,
otherwise $f$ is \emph{regular}.
It turns out that the regularity of the associated operator monotone function
guarantees the extendability of this metric to the complex projective space 
generated by the pure states (see \cite{Petz-Sudar}).


In the present paper, we introduce the concept of cancellability for operator connections in a natural way.
Various characterizations of cancellability with respect to operator monotone functions, Borel measures
and certain operator equations are provided.
It is shown that a connection is cancellable if and only if it is not a scalar multiple of trivial means.
Applications of this concept go to certain operator equations involving operator means.
We investigate the existence and the uniqueness of such equations.
It is shown that such equations are always solvable if and only if $f$ is unbounded and $f(0)=0$ where
$f$ is the associated operator monotone function.
We also characterize the condition $f(0)=0$ for arbitrary connections
without assuming the symmetry. Such connection is said to be nonregular.

\section{Cancellability of connections}

Since each connection is a binary operation, we can define the concept of cancellability as follows.

\begin{defn}
    A connection $\sigma$ is said to be
    \begin{itemize}
        \item   \emph{left cancellable} if for each $A>0, B \geqs 0$ and $C \geqs 0$,
                \begin{align*}
                	A \sm B = A \sm C \implies B=C
                \end{align*}
        \item   \emph{right cancellable} if for each $A>0, B \geqs 0$ and $C \geqs 0$,
                \begin{align*}
                	B \sm A = C \sm A \implies B=C
                \end{align*}
        \item   \emph{cancellable} if it is both left and right cancellable.
    \end{itemize}
\end{defn}

\begin{lem} \label{lem: nonconstant OM}
    Every nonconstant operator monotone function from $\R^+$ to $\R^+$ is injective.
\end{lem}
\begin{proof}
	Let $f: \R^+ \to \R^+$ be a nonconstant operator monotone function.
	Suppose there exist  $b>a \geqs 0$ such that $f(a)=f(b)$.
	Since $f$ is monotone increasing (in usual sense), 
	$f(x)=f(a)$ for all $a \leqs x \leqs b$ and 
	$f(y) \geqs f(b)$ for all $y \geqs b$.
	Since $f$ is operator concave,
	$f$ is concave in usual sense and hence $f(x)=f(b)$ for all $x \geqs b$.
	The case $a=0$ contradicts the fact that $f$ is nonconstant.
	Consider the case $a>0$.
	The monotonicity of $f$ yields $f(x) \leqs f(a)$ for all $0 \leqs x \leqs a$.
	The differentiability of $f$ implies that $f(x)=f(a)$ for all $x \geqs 0$,
	again a contradiction.
\end{proof}

A similar result for this lemma under the restriction that $f(0)=0$
was obtained in \cite{Molnar}.
The left-cancellability of connections is characterized as follows.

\begin{thm} \label{thm: cancallability 1}
	Let $\sigma$ be a connection with representing function $f$ and associated measure $\mu$.
	Then the following statements are equivalent:
    \begin{enumerate}
        \item[(1)]   $\sigma$ is left cancellable ;
        \item[(2)]   for each $A \geqs 0$ and $B \geqs 0$, $I \sigma A = I \sigma B \implies A=B$ ;
        \item[(3)]   $\sigma$ is not a scalar multiple of the left-trivial mean ;
        \item[(4)]	$f$ is injective, i.e., $f$ is left cancellable in the sense that
        			\begin{align*}
        				f \circ g = f \circ h \:\implies\: g=h \quad ;
        			\end{align*}
        \item[(5)]	$f$ is a nonconstant function ;
        \item[(6)]	$\mu$ is not a scalar multiple of the Dirac measure $\delta_0$ at $0$.
    \end{enumerate}
\end{thm}
\begin{proof}
	 Clearly, (1) $\Rightarrow$ (2) $\Rightarrow$ (3) and (4) $\Rightarrow$ (5).
	 For each $k \geqs 0$, it is straightforward to show that
	 the representing function of the connection
	 \begin{align*}
	 	k \omega_l \;:\; (A,B) \mapsto kA
	 \end{align*}
	 is the constant function $f \equiv k$
	 and its associated measure is given by $k\delta_0$. 
	Hence, we have the implications (3) $\Leftrightarrow$ (5) $\Leftrightarrow$ (6).
	By Lemma \ref{lem: nonconstant OM}, we have (5) $\Rightarrow$ (4).
	
	(4) $\Rightarrow$ (2): Assume that $f$ is injective.
    Consider $A \geqs 0$ and $B \geqs 0$ such that
    $I \sigma A = I \sigma B$. Then $f(A) = f(B)$ by \eqref{eq: f(A) = I sm A}. 
    Since $f^{-1} \circ f (x) =x$ 
    for all $x \in \R^+$, we have $A=B$.
	
    (2) $\Rightarrow$ (1): Let $A>0, B\geqs 0$ and $C \geqs 0$ be such that $A \sm B = A \sm C$.
    By the congruence invariance of $\sigma$, we have
    \begin{align*}
        A^{\frac{1}{2}} (I \sm A^{-\frac{1}{2}} B A^{-\frac{1}{2}}) A^{\frac{1}{2}} 
        \:=\: A^{\frac{1}{2}} (I \sm A^{-\frac{1}{2}} C A^{-\frac{1}{2}}) A^{\frac{1}{2}}
    \end{align*}
    and thus
    $I \sm A^{-\frac{1}{2}} B A^{-\frac{1}{2}} = I \sm A^{-\frac{1}{2}} C A^{-\frac{1}{2}} $. 
    Now, the assumption (2) implies
    $$ A^{-\frac{1}{2}} B A^{-\frac{1}{2}} \:=\: A^{-\frac{1}{2}} C A^{-\frac{1}{2}}, $$ i.e., $B=C$.
\end{proof}

Recall that the \emph{transpose} of a connection $\sigma$
is the connection 
\begin{align*}
	(A,B) \;\mapsto B \sm A.
\end{align*}
If $f$ is the representing function of $\sigma$, then
the representing function of the transpose of $\sigma$ is given by
the \emph{transpose} of $f$ (see \cite{Kubo-Ando}), defined by
\begin{align*}
        x \mapsto xf(1/x), \quad x>0. 
\end{align*}
A connection is \emph{symmetric} if it coincides with its transpose.

\begin{thm}
	Let $\sigma$ be a connection with representing function $f$ and associated measure $\mu$.
	Then the following statements are equivalent:
    \begin{enumerate}
        \item[(1)]   $\sigma$ is right cancellable ;
        \item[(2)]   for each $A \geqs 0$ and $B \geqs 0$, $A \sigma I = B \sigma I \implies A=B$ ;
        \item[(3)]   $\sigma$ is not a scalar multiple of the right-trivial mean ;
        \item[(4)]	the transpose of $f$ is injective ;
        \item[(5)]	$f$ is not a scalar multiple of the identity function $x \mapsto x$ ;
        \item[(6)]	$\mu$ is not a scalar multiple of the Dirac measure $\delta_1$ at $1$.
    \end{enumerate}
\end{thm}
\begin{proof}
	It is straightforward to see that, for each $k\geqs 0$, 
	the representing function of the connection
	\begin{align*}
		k \omega_r \;:\; (A,B) \mapsto kB
	\end{align*}
	is the function $x \mapsto kx$ and its associated measure is given by
	$k\delta_1$.
	The proof is done by replacing $\sigma$ with its transpose in Theorem \ref{thm: cancallability 1}.
\end{proof}

\begin{remk}
	The injectivity of the transpose of $f$ does not imply the surjectivity of $f$.
	To see that, take $f(x)=(1+x)/2$. Then the transpose of $f$ is $f$ itself.
\end{remk}

The following results are characterizations of cancellability of connections.

\begin{cor}
	Let $\sigma$ be a connection with representing function $f$ and associated measure $\mu$.
	Then the following statements are equivalent:
    \begin{enumerate}
        \item[(1)]   $\sigma$ is cancellable ;
        \item[(2)]   $\sigma$ is not a scalar multiple of the left/right-trivial mean ;
        \item[(3)]	$f$ and its transpose are injective ;
        \item[(4)]	$f$ is neither a constant function nor a scalar multiple of the identity function;
        \item[(5)]	$\mu$ is not a scalar multiple of $\delta_0$ or $\delta_1$.
    \end{enumerate}
    In particular, every nontrivial mean is cancellable. 
\end{cor}

\begin{remk}
    The ``order cancellability" does not hold for general connections, 
    even if we restrict them to the class of means.
    For each $A,B>0$, it is not true that the condition $I \sm A \leqs I \sm B$
    or the condition $A \sm I \leqs B \sm I$ implies $A \leqs B$. 
    To see this, take $\sigma$ to be the geometric mean.
    It is not true that $A^{1/2} \leqs B^{1/2}$ implies $A \leqs B$ in general.
\end{remk}

\section{Applications to certain operator equations}

Cancellability of connections can be restated in terms of the uniqueness of certain operator
equations as follows.
A connection $\sigma$ is left cancellable if and only if 
\begin{quote}
	for each given $A>0$ and $B \geqs 0$, if the equation $A \sm X =B$
has a solution, then it has a unique solution.
\end{quote}
The similar statement for right-cancellability holds.
In this section, we investigate the existence and the uniqueness of the operator equation 
$A \sm X =B$.


\begin{thm} \label{thm: operator eq 1}
	Let $\sigma$ be a connection which is not a scalar multiple of the left-trivial mean.
	Let $f$ be its representing function.
	Given $A>0$ and $B \geqs 0$, the operator equation
	\begin{align*}
		A \, \sigma X \, \:=\: B 	
	\end{align*}
	has a (positive) solution if and only if $\Sp (A^{-\frac{1}{2}} B A^{-\frac{1}{2}}) 
	\subseteq \Range (f)$.
	In fact, such a solution is unique and given by
	\begin{align*}
		X \:=\: A^{\frac{1}{2}} f^{-1}(A^{-\frac{1}{2}} B A^{-\frac{1}{2}}) A^{\frac{1}{2}}.
	\end{align*}
\end{thm}
\begin{proof}
	Suppose that there is a positive operator $X$ such that $A \sm X =B$.
	The congruence invariance of $\sigma$ yields
	\begin{align*}
		A^{\frac{1}{2}} (I \sm A^{-\frac{1}{2}} X A^{-\frac{1}{2}}) A^{\frac{1}{2}} \:=\: B.
	\end{align*}
	The property \eqref{eq: f(A) = I sm A} now implies
	\begin{align*}
		f(A^{-\frac{1}{2}} X A^{-\frac{1}{2}}) 
		\:=\: I \sm A^{-\frac{1}{2}} X A^{-\frac{1}{2}}\:=\: A^{-\frac{1}{2}} B A^{-\frac{1}{2}}.
	\end{align*}
	By spectral mapping theorem,
	\begin{align*}
		\Sp (A^{-\frac{1}{2}} B A^{-\frac{1}{2}} ) 
		=\Sp\left(f(A^{-\frac{1}{2}} X A^{-\frac{1}{2}}) \right)
		= f\left(\Sp(A^{-\frac{1}{2}} X A^{-\frac{1}{2}})\right)
		\subseteq	\Range (f).
	\end{align*}
	Conversely, suppose that $\Sp (A^{-\frac{1}{2}} B A^{-\frac{1}{2}}) 
	\subseteq \Range (f)$.
	Since $\sigma \neq k \omega_l$ for all $k \geqs 0$, we have that $f$ is nonconstant
	by Theorem \ref{thm: cancallability 1}. 
	It follows that $f$ is injective by Lemma \ref{lem: nonconstant OM}.
	The assumption yields the existence of the operator 
	$X \equiv A^{\frac{1}{2}} f^{-1}(A^{-\frac{1}{2}} B A^{-\frac{1}{2}}) A^{\frac{1}{2}}$.
	We obtain from the property \eqref{eq: f(A) = I sm A} that
	\begin{align*}
		A \sm X 
		\:&=\:	A \sm  A^{\frac{1}{2}} f^{-1}(A^{-\frac{1}{2}} B A^{-\frac{1}{2}}) A^{\frac{1}{2}} \\
		\:&=\:	A^{\frac{1}{2}} \left[I \sm f^{-1}(A^{-\frac{1}{2}} B A^{-\frac{1}{2}}) \right]
		 		A^{\frac{1}{2}} \\
		\:&=\:	A^{\frac{1}{2}} f\left(f^{-1}(A^{-\frac{1}{2}} B A^{-\frac{1}{2}}\right) 
				A^{\frac{1}{2}} \\
		\:&=\:	B
	\end{align*}
	The uniqueness of a solution follows from the left-cancellability of $\sigma$.
\end{proof}

Similarly, we have the following theorem.

\begin{thm}
	Let $\sigma$ be a connection which is not a scalar multiple of the right-trivial mean.
	Given $A>0$ and $B \geqs 0$, the operator equation
	\begin{align*}
		X \,\sigma A \, \:=\: B	
	\end{align*}
	has a (positive) solution if and only if $\Sp (A^{-1/2} B A^{-1/2}) \subseteq \Range (g)$,
	here $g$ is the representing function of the transpose of $\sigma$.
	In fact, such a solution is unique and given by
	\begin{align*}
		X \:=\: A^{1/2} g^{-1}(A^{-1/2}BA^{-1/2}) A^{1/2}.
	\end{align*}
\end{thm}

\begin{thm} \label{cor: f is unbounded and f(0)=0}
	Let $\sigma$ be a connection 
	with representing function $f$.
	Then the following statements are equivalent:
	\begin{enumerate}
		\item[(1)]	$f$ is unbounded and $f(0)=0$ ; 
		\item[(2)]	$f$ is surjective, i.e., $f$ is right cancellable in the sense that
        			\begin{align*}
        				g \circ f = h \circ f \:\implies\: g=h \quad ;
        			\end{align*}
		\item[(3)]	the operator equation
					\begin{align}
						A \,\sigma X \, \:=\: B		\label{eq: A sm X =B}
					\end{align}
					has a unique solution for any given  $A>0$ and $B \geqs 0$.
	\end{enumerate}
	Moreover, if (1) holds, then
	the solution of \eqref{eq: A sm X =B} varies continuously in each given $A>0$
	and $B \geqs 0$, i.e. the map $(A,B) \mapsto X$ is separately continuous with respect to
	the strong-operator topology.
\end{thm}
\begin{proof}
	(1) $\Rightarrow$ (2): This follows directly from the intermediate value theorem.
	
	(2) $\Rightarrow$ (3): It is immediate from Theorem \ref{thm: operator eq 1}.
	
	(3) $\Rightarrow$ (1): Assume (3). The uniqueness of solution for the equation $A \sm X =B$
	implies the left-cancellability of $\sigma$.
	By Theorem \ref{thm: cancallability 1}, $f$ is injective.
	The assumption (3) implies the existence of a positive operator $X$ such that
	\begin{align*}
		f(X) \:=\: I \sm X \:=\: 0.
	\end{align*}
	The spectral mapping theorem implies that $f(\ld)=0$ for all $\ld \in \Sp(X)$.
	Since $f$ is injective, we have $\Sp(X) =\{\ld\}$ for some $\ld \in \R^+$.
	Suppose that $\ld>0$. Then there is $\epsilon >0$ such that $X > \ep I >0$.
	It follows from the monotonicity of $\sigma$ that
	\begin{align*}
		0 \:=\: I \sm X  \:\geqs\: I \sm \ep I \:=\: f(\ep)I,
	\end{align*}
	i.e. $f(\ep)=0$. Hence, $\ep = \ld$. Similarly, since $X > (\ep/2) I >0$, we have
	$\ep = 2 \ld$, a contradiction. Thus, $\ld=0$ and $f(0)=0$.
	
	Now, let $k>0$. The assumption (3) implies the existence of $X \geqs 0$ such that
	$I \sm X =kI$. Since $f(X)=kI$, we have $f(\ld)=k$ for all $\ld \in \Sp(X)$.
	Since $\Sp(X)$ is nonempty, there is $\ld \in \Sp(X)$ such that $f(\ld)=k$.
	Therefore, $f$ is unbounded.
	
	Assume that (1) holds. Then the map $(A,B) \mapsto X$ is well-defined.
	Recall that if $A_n \in B(\HH)^+$ converges strongly to $A$, 
	then $\phi(A_n)$ converges strongly to $\phi(A)$
	for any continuous function $\phi$.
	It follows that the map 
	$$(A,B) \mapsto X = A^{\frac{1}{2}} f^{-1}(A^{-\frac{1}{2}} B A^{-\frac{1}{2}}) A^{\frac{1}{2}} $$ 
	is separately continuous in each variable.
\end{proof}


\begin{ex}

Consider the quasi-arithmetic power mean $\#_{p,\ap}$
with exponent $p \in [-1,1]$ and weight $\ap \in (0,1)$,
defined by
\begin{align*}
	A \,\#_{p,\ap}\, B \:=\: \left[(1-\ap)A^p + \ap B^p \right]^{1/p}.
\end{align*}
Its representing function of this mean is given by
\begin{align*}
	f_{p,\ap}(x) \:=\: (1-\ap+\ap x^p)^{1/p}.
\end{align*}
The special cases $p=1$ and $p=-1$ are the $\ap$-weighted arithmetic mean
and the $\ap$-weighted harmonic mean, respectively.
The case $p=0$ is defined by continuity and, in fact, $\#_{0,\ap} = \#_{\ap}$ and 
$f_{0,\ap}(x)=x^{\ap}$.
Given $A>0$ and $B \geqs 0$, consider the operator equation
	\begin{align}
		A \,\#_{p,\ap}\, X \:=\: B.  	\label{eq weighted GM}
	\end{align}

\underline{The case $p=0$ :}  
	Since the range of $f_{0,\ap}(x)=x^{\ap}$ is $\R^+$, 
	the equation \eqref{eq weighted GM} always has a unique solution given by
	\begin{align*}
		X 
		  \:=\: A^{1/2} (A^{-1/2}BA^{-1/2})^{1/\ap} A^{1/2} 
		\:\equiv \: A \,\#_{1/\ap}\, B.
	\end{align*}

\underline{The case $0<p\leqs 1$ :} 
The range of $f_{p,\ap}$ is the interval $[(1-\ap)^{1/p}, \infty)$.
Hence, the equation \eqref{eq weighted GM} is solvable if and only if
$\Sp(A^{-1/2} B A^{-1/2}) \subseteq [(1-\ap)^{1/p}, \infty)$, i.e., $B \geqs (1-\ap)^{1/p} A$.


\underline{The case $-1 \leqs p< 0$ :} 
The range of $f_{p,\ap}$ is the interval $[0, (1-\ap)^{1/p})$.
Hence, the equation \eqref{eq weighted GM} is solvable if and only if
$\Sp(A^{-1/2} B A^{-1/2}) \subseteq [0, (1-\ap)^{1/p})$, i.e., $B < (1-\ap)^{1/p} A$.


For each $p \in [-1,0) \cup (0,1]$ and $\ap \in (0,1)$, we have
\begin{align*}
	f_{p,\ap}^{-1} (x) \:=\:  \left(1-\frac{1}{\ap} + \frac{1}{\ap} x^p \right)^{1/p}.
\end{align*}
Hence, the solution of \eqref{eq weighted GM} is given by
\begin{align*}
	X \:=\: \left[(1-\frac{1}{\ap})A^p + \frac{1}{\ap} B^p \right]^{1/p} 
	\:\equiv\: A \;\#_{p, \frac{1}{\ap}}\; B.
\end{align*}
\end{ex}

\begin{ex}
	Let $\sigma$ be the logarithmic mean with representing function
	\begin{align*}
		f(x) \:=\: \frac{x-1}{\log x}, \quad x>0.
	\end{align*}
	Here, $f(0) \equiv 0$ by continuity. We see that $f$ is unbounded. 
	Thus, the operator equation $A \,\sigma X \, =B$ is solvable for all $A>0$ and $B \geqs 0$.
\end{ex}

\begin{ex}
	Let $\eta$ be the dual of the logarithmic mean, i.e.,
	\begin{align*}
		\eta \,:\, (A,B) \mapsto \left(A^{-1} \, \sigma \, B^{-1}\right)^{-1}
	\end{align*}
	where $\sigma$ denotes the logarithmic mean.
	The representing function of $\eta$ is given by
	\begin{align*}
		f(x) \:=\: \frac{x}{x-1} \log{x}, \quad x>0.
	\end{align*}
	We have that $f(0) \equiv 0$ by continuity and $f$ is unbounded. 
	Therefore, the operator equation $A \,\eta X \, =B$ is solvable for all $A>0$ and $B \geqs 0$.
\end{ex}

\section{Regularity of connections}

In this section, we consider the regularity of connections.

\begin{thm} \label{thm: regularity}
	Let $\sigma$ be a connection with representing function $f$ and associated measure $\mu$.
	Then the following statements are equivalent.
	\begin{enumerate}
		\item[(1)]	$f(0)=0$ ;
		\item[(2)]	$\mu(\{0\}) = 0$ ;
		\item[(3)]	$I \sm 0 = 0$ ;
		\item[(4)]	$A \sm 0 =0$ for all $A \geqs 0$ ;
		\item[(5)]	for each $A \geqs 0$, $0 \in \Sp(A) \Longrightarrow 0 \in \Sp(I \sm A)$ ; 
		\item[(6)]	for each $A,X \geqs 0$, $0 \in \Sp(A) \Longrightarrow 0 \in \Sp(X \sm A)$. 
	\end{enumerate}
\end{thm}
\begin{proof}
	From the integral representation \eqref{int rep of OMF}, we have 
	\begin{align}
		f(x) \:=\: \mu(\{0\}) + \mu(\{1\})x + \int_{(0,1)} (1\,!_t\, x)\,d\mu(t), \quad x\geqs 0,
		\label{eq: int rep of OMR expand}
	\end{align}
	i.e. $f(0)=\mu(\{0\})$. From the property \eqref{eq: f(A) = I sm A}, we have
	$I \sm 0 =f(0)I$. Hence, (1)-(3) are equivalent. 
	It is clear that (4) $\Rightarrow$ (3) and (6) $\Rightarrow$ (5).
	
	(3) $\Rightarrow$ (4): Assume that $I \sm 0=0$. For any $A>0$,
	we have by the congruence invariance that
	\begin{align*}
		A \sm 0 \:=\: A^{\frac{1}{2}} (I \sm 0) A^{\frac{1}{2}} \:=\: 0.
	\end{align*}
	For general $A \geqs 0$, 
	we have $(A + \ep I) \sm 0 =0$ for all $\ep>0$ by the previous claim 
	and hence $A \sm 0 =0$ by the continuity from above.
	 
	
	(5) $\Rightarrow$ (1): We have $0 \in \Sp(I \sm 0) = \Sp(f(0)I) = \{f(0)\}$, i.e. $f(0)=0$.
	
	(1) $\Rightarrow$ (6): Assume $f(0)=0$. Consider $A \geqs 0$ such that $0 \in \Sp(A)$,
	i.e. $A$ is not invertible.
	Assume first that $X>0$. Then
	\begin{align*}
		X \sm A 
		\:=\: X^{\frac{1}{2}} (I \sm X^{-\frac{1}{2}} A X^{-\frac{1}{2}}) X^{\frac{1}{2}}
		\:=\: X^{\frac{1}{2}} f(X^{-\frac{1}{2}} A X^{-\frac{1}{2}}) X^{\frac{1}{2}}.
	\end{align*}
	Since $X^{-\frac{1}{2}} A X^{-\frac{1}{2}}$ is not invertible, we have 
	$0 \in \Sp(X^{-\frac{1}{2}} A X^{-\frac{1}{2}})$ and hence by spectral mapping theorem
	\begin{align*}
		0 \:=\: f(0) 
		\in f \left(\Sp(X^{-\frac{1}{2}} A X^{-\frac{1}{2}}) \right)
		\:=\: \Sp\left(f(X^{-\frac{1}{2}} A X^{-\frac{1}{2}}) \right).
	\end{align*}
	This implies that $X \sm A$ is not invertible.
	Now, consider $X \geqs 0$. The previous claim shows that $(X+I)\sm A$ is not invertible.
	Since $X \sm A \leqs (X+I) \sm A$, we conclude that $X \sm A$ is not invertible.
\end{proof}

We say that a connection $\sigma$ is \emph{nonregular} if one (thus, all) of the conditions
in Theorem \ref{thm: regularity} holds, otherwise $\sigma$ is \emph{regular}.
Hence, regular connections correspond to regular operator monotone functions and regular monotone metrics.

\begin{remk}
Recall from \cite{Pattrawut} that every connection $\sigma$ can be written as
the sum of three connections.
The \emph{singularly discrete part} 
is a countable sum of weighted harmonic means with nonnegative coefficients.
The \emph{absolutely continuous part} 
is a connection admitting an integral representation
with respect to Lebesgue measure $m$ on $[0,1]$.
The \emph{singularly continuous part} 
is associated to a continuous measure mutually singular to $m$.
Hence, a connection, whose associated measure has no singularly discrete part, is nonregular.
\end{remk}

\begin{remk}
	Let $\sigma$ be a connection with representing function $f$ and associated measure $\mu$.
	Let $g$ be the representing function of the transpose of $\sigma$.
	From \eqref{eq: int rep of OMR expand},
	\begin{align*}
		g(0) = \lim_{x \to 0^+} xf(\dfrac{1}{x}) = \lim_{x \to \infty} \frac{f(x)}{x}
		= \mu(\{1\}).
	\end{align*}
	Thus, the transpose of $\sigma$ is nonregular if and only if $\mu(\{1\})=0$.
\end{remk}

\begin{thm} \label{thm: regularity of mean}
	The following statements are equivalent for a mean $\sigma$.
	\begin{enumerate}
		\item[(1)]	$\sigma$ is nonregular ;
		\item[(2)]	$I \sm P = P$ for each projection $P$.
	\end{enumerate}
\end{thm}
\begin{proof}
	(1) $\Rightarrow$ (2): Assume that $f(0)=0$ and consider a projection $P$.
	Since $f(1)=1$, we have  $f(x)=x$ for all $x \in \{0,1\}\supseteq \Sp(P)$.
	Thus $I \sm P = f(P)=P$.
	
	(2) $\Rightarrow$ (1):	We have $0=I \sm 0=f(0)I$, i.e. $f(0)=0$.
\end{proof}

To prove the next result, recall the following lemma.

\begin{lem}[\cite{Kubo-Ando}]
	If $f: \R^+ \to \R^+$ is an operator monotone function such that $f(1)=1$
	and $f$ is neither the constant function $1$ nor the identity function, then
	\begin{enumerate}
		\item[1)]	$0<x<1 \implies x<f(x)<1$
		\item[2)]	$1<x \implies 1<f(x)<x$.
	\end{enumerate}
\end{lem}

\begin{thm}
	Let $\sigma$ be a nontrivial mean.
	For each $A \geqs 0$, if $I \sm A =A$, then $A$ is a projection.
	Hence, the following statements are equivalent:
\begin{enumerate}
	\item[(1)]	$\sigma$ is nonregular ;
	\item[(2)]	for each $A \geqs 0$, $A$ is a projection if and only if $I \sm A = A$.
\end{enumerate}
\end{thm}
\begin{proof}
	Since $I \sm A=A$, we have $f(A)=A$ by \eqref{eq: f(A) = I sm A}.
	Hence $f(x)=x$ for all $x \in \Sp(A)$ by the injectivity of the continuous functional calculus.
	Since $\sigma$ is a nontrivial mean, the previous lemma forces that $\Sp(A) \subseteq \{0,1\}$,
	i.e. $A$ is a projection.
\end{proof}

\begin{thm}
	Under the condition that $\sigma$ is a left-cancellable connection
	with representing function $f$, the following
	statements are equivalent:
	\begin{enumerate}
		\item[(1)]	$\sigma$ is nonregular.
		\item[(2)]	The equation $f(x)=0$ has a solution $x$.
		\item[(3)]	The equation $f(x)=0$ has a unique solution $x$.
		\item[(4)]	The only solution to $f(x)=0$ is $x=0$.
		\item[(5)]	For each $A>0$, the equation $A \sm X =0$ has a solution $X$.
		\item[(6)]	For each $A>0$, the equation $A \sm X =0$ has a unique solution $X$.
		\item[(7)]	For each $A>0$, the only solution to the equation $A \sm X =0$ is $X=0$.
		\item[(8)]	The equation $I \sm X =0$ has a solution $X$.
		\item[(9)]	The equation $I \sm X =0$ has a unique solution $X$.
		\item[(10)]	The only solution to the equation $I \sm X =0$ is $X=0$.
	\end{enumerate}
	Similar results for the case of right-cancellability hold.
\end{thm}
\begin{proof}
	It is clear that (1) $\Rightarrow$ (2), (6) $\Rightarrow$ (5) and (10) $\Rightarrow$ (8).
	Since $f$ is injective by Theorem \ref{thm: cancallability 1}, 
	we have (2) $\Rightarrow$ (3) $\Rightarrow$ (4).
		
	(4)	$\Rightarrow$ (10):	Let $X \geqs 0$ be such that $I \sm X =0$. 
	Then $f(X)=0$ by \eqref{eq: f(A) = I sm A}.
	By spectral mapping theorem, $f(\Sp(X)) =\{0\}$. Hence, $\Sp(X)=\{0\}$, i.e. $X=0$.
	
	(8)	$\Rightarrow$ (9): Consider $X \geqs 0$ such that $I \sm X =0$. Then $f(X)=0$.
	Since $f$ is injective with continuous inverse, we have $X=f^{-1}(0)$.
	
	(9)	$\Rightarrow$ (6):	Use congruence invariance.
	
	(5)	$\Rightarrow$ (7):	Let $A>0$ and consider $X \geqs 0$ such that $A \sm X=0$.
	Then $A^{\frac{1}{2}} (I \sm A^{-\frac{1}{2}} X A^{-\frac{1}{2}}) A^{\frac{1}{2}} =0$,
	i.e. $f(A^{-\frac{1}{2}} X A^{-\frac{1}{2}}) = I \sm A^{-\frac{1}{2}} X A^{-\frac{1}{2}} =0$.
	Hence,
	\begin{align*}
		f\left( \Sp(A^{-\frac{1}{2}} X A^{-\frac{1}{2}}) \right)
		= \Sp\left( f(A^{-\frac{1}{2}} X A^{-\frac{1}{2}}) \right)
		=\{0\}.
	\end{align*}
	Suppose there exists $\ld \in \Sp(A^{-\frac{1}{2}} X A^{-\frac{1}{2}})$
	such that $\ld>0$.
	Then $f(0)<f(\ld)=0$, a contradiction.
	Hence, $\Sp(A^{-\frac{1}{2}} X A^{-\frac{1}{2}}) =\{0\}$, i.e.
	$A^{-\frac{1}{2}} X A^{-\frac{1}{2}} =0$ or $X=0$.
	
	(7) $\Rightarrow$ (1): We have $f(0)I = I \sm 0 =0$, i.e. $f(0)=0$.
\end{proof}

\end{document}